\documentclass[amssymb,11pt,leqno]{amsart}

\usepackage{amsfonts}
\usepackage{amssymb}
\usepackage{amsthm}
\usepackage{newlfont}
\usepackage{amsrefs}

\newcommand{\R}{\mathbb{R}}
\newcommand{\C}{\mathbb{C}}

\newcommand{\N}{\mathbb{N}}
\newcommand{\E}{\mathbb{E}}

\newcommand{\test}{\mathcal S}
\newcommand{\M}{\mathcal M}

\def\vol{\mbox{\rm Vol}}

\def\a{\alpha}
\def\b{\beta}

\def\ph{\phi}
\def\O{\Omega}

\newtheorem{Thm}{Theorem}[section]

\newtheorem{lm}[Thm]{Lemma}
\newtheorem{prop}[Thm]{Proposition}

\theoremstyle{definition}

\usepackage{graphicx}
\usepackage{latexsym}
\usepackage{enumerate}
\numberwithin{equation}{section}
\def\Re{\text{Re }}
\def\tfrac{\textstyle\frac}

\def\dim{\text{dim }}

\begin{document}

\title{Positive definite distributions and normed spaces }

\author{N.J. Kalton}
\address{Department of Mathematics\\ University of Missouri\\ Columbia\\ Missouri 65211}
\email{kaltonn@missouri.edu}

\author{M. Zymonopoulou}\address{ Department of Mathematics\\Case Western Reserve University\\ Cleveland}
\email{marisa.zym@gmail.com}

\thanks{The first author acknowledges the support of NSF grant DMS-0555670. The second author was partially supported by the NSF grant DMS-0652722}

\begin{abstract}  We answer a question of Alex Koldobsky. We show that for each $-\infty<p<2$ and each $n\ge 3-p$ there is a normed space $X$ of dimension $n$ which embeds in $L_s$ if and only if $-n<s\le p.$  \end{abstract}

\maketitle{}




                                        \section{Introduction}\label{intro}

 Let $\|\cdot\|$ be a norm on $\mathbb R^n.$
  It is well-known that if $p>0$ and not an even integer then $X=(\mathbb R^n,\|\cdot\|)$ embeds isometrically into $L_p$ if and only $\Gamma(-p/2)\|\cdot\|^p$ is a positive definite distribution (see \cite{Koldobsky2005} Theorem 6.10).
  In \cite{Koldobsky1999b} this idea was extended to the case when $p<0.$ Let $\mathcal S(\mathbb R^n)$ denote the Schwartz class of the rapidly decreasing functions on $\mathbb R^n.$ If $p<0$ and $n+p>0$ then the function $\|x\|^p$ is locally integrable and we say that $X$ embeds (isometrically) into $L_p$ if the distribution
  $\|\cdot\|^p$ is positive definite, i.e. for every non-negative even test function $\phi\in\mathcal S(\mathbb R^n),$
  $$\langle (\|\cdot\|^p)^{\wedge},\phi\rangle\ge 0.$$
This can be expressed in the following form: We say that  $X=(\R^n,\|\cdot\|)$ embeds into $L_{p}$, where $p<0<p+n,$ if there exists a finite Borel measure $\mu$ on $S^{n-1}$ so that for every even test function $\ph\in\test(\R^n)$
\begin{equation}\label{Def:p<0}\int_{\R^n}\|x\|^{p}\ph(x)dx=\int_{S^{n-1}}\Bigl(\int_0^{\infty}t^{-p-1}\hat{\ph}(t\xi)dt\Bigr)d\mu (\xi).
\end{equation}
Later in \cite{KaltonKoldobskyYaskinYaskina2007} the appropriate definition for $p=0$ was explored: a normed space $X$ embeds into $L_0$ if and only $-\ln \|x\|$ is positive definite outside of the origin of $\R^n$.

Part of the motivation for this definition is its connection to intersection bodies.  The class of intersection bodies was defined by Lutwak \cite{Lutwak1988} and played an important role to the solution of the Busemann-Petty problem. Let $K$ and $L$ two origin symmetric star bodies in $\R^n.$ We say that $K$ is the {\it intersection body of $L$} if the radius of $K$ in every direction is equal to the volume of the central hyperplane section of $L$ perpendicular to this direction, i.e. for every $\xi\in S^{n-1},$
$$\|\xi\|_K^{-1}=\vol _{n-1}(L\cap \xi^{\perp}),$$
where $\|x\|_K=\min \{a\geq 0 :x\in aK\},$ is the Minkowski functional of $K$. Note that if $K$ is convex then $\|\cdot\|_K$ is a norm. The class of {\it intersection bodies} is defined as the closure, in the radial metric, of the set of intersection bodies of all star bodies.
This class was extended in \cite{Koldobsky1999} and \cite{Koldobsky2000}, to  the class of $k$-intersection bodies, where $k\in\mathbb N$. Koldobsky in \cite{Koldobsky2000} showed that $X$ embeds into $L_{-k}$ if and only if its unit ball is a $k$-intersection body. For more on $k$-intersection bodies see \cite{Koldobsky2005}, (Chapters 4 and 6) or \cite{KoldobskyYaskin2008}, (Chapters 6 and 7).

If $n>-p$, we denote by $\mathcal I_p(n)$ the collection of the finite-dimensional Banach spaces $X$ of dimension $n$ which embed into $L_p$ where $-\infty<p<\infty;$ we will adopt the convention that $\mathcal I_p(n)=\mathcal B_n,$ the collection of all spaces of dimension $n$ when $n\le -p.$  It was shown by Koldobsky \cite{Koldobsky1999b} that if $p\le 3-n$ then $\mathcal I_p(n)=\mathcal B_n.$ Let $\mathcal I_p=\cup_{n\in\mathbb N}\mathcal I_p(n).$  A classical result of Bretagnolle, Dacunha-Castelle and Krivine \cite{Bretagnolleetal1965/6} shows that if $0<p\le q\le 2$ then $\mathcal I_q\subset\mathcal I_p.$  Combining results of \cite{KaltonKoldobskyYaskinYaskina2007} and \cite{Koldobsky1999b} gives that $\mathcal I_q\subset \mathcal I_p$ where $q\in [0,2]$ and $p\le q.$  It is, however, an open problem whether the same is true when $q<0.$  E.Milman \cite{MilmanE2006} showed that if $m\in\mathbb N$ and $p<0$ then $\mathcal I_{p}\subset\mathcal I_{mp}.$

A second problem in this area is to establish whether the classes $\mathcal I_p(n)$ for $-\infty<p\le 1$ are really distinct (see for example \cite{KoldobskyYaskin2008} p.99). In this article we give a complete answer to this question. Previously only some partial results have been established.  For the case $0<p\le 1,$
 it is shown  in \cite{KaltonKoldobsky2004} that if $0<p<s\le 1$ then $\mathcal I_p\neq \mathcal I_s.$
 However the methods of \cite{KaltonKoldobsky2004} are infinite-dimensional and only show that for given $0<p<q\le 1$ we have $\mathcal I_p(n)\neq\mathcal I_q(n)$ for some $n=n(p,q).$
It was noted in \cite{KaltonKoldobskyYaskinYaskina2007} that the space $\mathbb R\oplus_2\ell_1^n$ belongs to $\mathcal I_0$ for all $n$ but for each $p>0$ there is an $n\in\N$ so that $\mathbb R\oplus_2\ell_1^n\notin \mathcal I_p.$
In the case where $p,q<0$ it is clear that if $p \le 3-n<q$ then $\mathcal I_q(n)$ is strictly contained in $\mathcal I_p(n)=\mathcal B_n.$  In fact $\ell_s^n\notin \mathcal I_q(n)$ if $2<s\le \infty$ (see \cite{Koldobsky2005} Theorem 4.13 or \cite{Koldobsky1998a}).  For other values of $n$, there are some recent partial
 results.
In \cite{Schlieper2007} it was shown that $\mathcal I_{-4}(n)\setminus \mathcal I_{-2}(n)\neq \emptyset$ for all $n\ge 7$ (and hence for $n\ge 5$) and that $\mathcal I_{-1/3}(n)\setminus\mathcal I_{-1/6}(n)\neq \emptyset$ for all $n\ge 4.$   More recently Yaskin \cite{Yaskin2008} showed that if $l<k$ are integers and $k>3-n$ then $\mathcal I_l(n)\setminus \mathcal I_k(n)\neq \emptyset.$

Our main example is that if $X=\ell_2^m\oplus_r\ell_q^n$ where $1\le q<r\le 2$ and $n\ge 2$ then $X\in\mathcal I_p$ if and only if $p\le q-m.$  Thus it follows immediately that if $p\in (3-n,0)$ there exists a normed space $X$ so that $X\in\mathcal I_p(n)$ but for every $q>p$ $X\notin\mathcal I_q(n).$ Note that even in the case when $0<p<1$ this improves considerably the results in \cite{KaltonKoldobsky2004} and the examples are much more natural.

To obtain these results we prove a general result on absolute direct sums of normed spaces.
Let $X$ and $Y$ denote two finite-dimensional Banach spaces.
 Let $N$ be any absolute norm on $\R^2,$ ie. $N(x,y)=N(|x|,|y|),$ satisfying the normalization property $N(1,0)=N(0,1)=1.$ We consider the absolute $N$-direct sum of $X$ and $Y$, denoted $X\oplus_N Y$ that is defined as the space of pairs $\{(x,y),x\in X, y\in Y\}$ equipped with the norm $N.$
 $$ \|(x,y)\|=N(\|x\|_X,\|y\|_Y), \qquad x\in X,\ y\in Y.$$
In the special case where $N(x,y)=(x^r+y^r)^{1/r},$ we write $X\oplus_NY=X\oplus_r Y.$

We examine the situation when $X\oplus_NY\in\mathcal I_p.$ There is an earlier result of Koldobsky of this type; see \cite{Koldobsky2005}, Theorem 4.21 or \cite{Koldobsky1998}.  Koldobsky shows that if $p<0<2<q$ and $X\oplus_qY\in\mathcal I_p$  with $\dim Y\ge 1$ then $\dim X\le 2-p.$ In fact this results hold under the more general hypothesis if $p<2<q$.

A typical result we prove  is that if $r\le 2$ and $X\oplus_rY\in\mathcal I_p$ where $p\le 2$ then $X\in \mathcal I_q$ as long as $p\le q\le m+p$ where $m=\dim Y.$   We consider a more general absolute norm $N$  and use functional analytic and probabilistic methods as well as the theory of Gaussian processes, rather than the usual distributional approach from \cite{Koldobsky2005} or \cite{KoldobskyYaskin2008}.

The remainder of the paper is devoted to showing that the examples   $X=\ell_2^m\oplus_r\ell_q^n$ where $1\le q<r\le 2$ and $n\ge 2$ belong to $\mathcal I_p$ if $p\le q-m.$  This requires a probabilistic approach using stable random variables.

\section{Gaussian embeddings}

Throughout this paper,
$(\Omega,\mu)$ will be a Polish space with a  $\sigma-$finite Borel measure and $\mathcal M(\Omega,\mu)$ will be the space of all real-valued measurable functions on $\Omega.$  In the special case when $\mu(\Omega)=1$ we say that $\mu$ is a probability measure and the members of $\mathcal M(\Omega,\mu)$ are then called {\it random variables}.
Let $X$ be a finite dimensional normed space and suppose $T:X\to \mathcal M(\Omega,\mu)$ is a linear map. Suppose $0<p<\infty$.  We shall say that $T$ is a {\it $c$-standard embedding} of $X$ into $L_p(\Omega,\mu)$, where $c>0,$ if
$$\|x\|^p=\frac{1}{ c^p}\int_{\Omega}|Tx|^p\,d\mu, \qquad x\in X.$$

Let $(\Omega',\mathbb P)$ be some probability space.  A measurable map $\xi:\Omega'\to  X$ is called an $X$-valued Gaussian process if it takes the form
$$ \xi=\sum_{j=1}^m \gamma_jx_j$$ where $x_1,\ldots,x_m\in X$ and $\{\gamma_1,\ldots,\gamma_m\}$ is a sequence of independent normalized Gaussians.
The {\it rank} of $\xi$ is defined to be the dimension of the space spanned by $\{x_1,\ldots,x_m\}$; we say that $\xi$ has {\it full rank} if its rank is equal to the dimension of $X.$

Suppose $-\infty<p<\infty$ and $X$ has dimension $n>-p$. A linear map $T:X\to\mathcal M(\Omega,\mu)$ is called a {\it $c$-Gaussian embedding} of $X$ into $L_p(\Omega,\mu)$ if
\begin{equation}\label{Gaussiandef} \mathbb E\|\xi\|^p=\frac{1}{c^p}\int_{\Omega}(\sum_{j=1}^n(Tx_j)^2)^{p/2}\,d\mu\end{equation} whenever $\xi$ is an $X$-valued Gaussian random variable of full rank.  In fact it can be shown quite easily that
\eqref{Gaussiandef} holds for all $\xi$ of rank greater than $-p.$  It should be noted that if $p\le -1$ it is not generally true that $\int|Tx|^p<\infty$ for each $x\in X.$

It will be important for us that the existence of a Gaussian embedding in $L_p$ in the case when $p<0$ is equivalent to the fact that $X\in\mathcal I_p$ according to the definition in \cite{Koldobsky1999b} via positive definite functions (see \eqref{Def:p<0}).  One direction of this equivalence appears implicitly in \cite{KaltonKoldobsky2005} but the converse direction has not apparently appeared before, although it has been known for a number of years.

We first need a preparatory Lemma. Let $g_a$ denote the density function $$g_a(x)=(2\pi)^{-n/2} a^{-n}e^{-|x|^2/2a^2}, \qquad x\in\R^n$$ For $y\in Y$ we define $h_y(x)=(x,y).$ If $f\in\mathcal S(\mathbb R^n)$ we denote by $\tau_yf$ the function $\tau_yf(x)=f(x-y).$

\begin{lm}\label{prep} Suppose $n\in\mathbb N$ and $\rho\in\mathcal S'(\mathbb R^n)$ is such that $\langle e^{-(Ax,x)},\rho\rangle=0$ for every positive definite matrix $A$.  Then for $a>0$ and fixed $y\in\mathbb R^n$,  we have $$\langle \tau_yg_a+\tau_{-y}g_a,\rho\rangle=0, \qquad y\in\mathbb R^n.$$\end{lm}

\begin{proof}  We start with two observations about the case $n=1.$  First we observe that the map $\{z:\ \Re z >0\}\to \mathcal S(\mathbb R)$ defined by $z\mapsto e^{-zx^2/2}$ is analytic into the locally convex Fr\'echet space $\mathcal S(\mathbb R).$  Similarly so is the map $\mathbb C\to\mathcal S(\mathbb R)$ defined by $z\mapsto e^{-a^2(x^2+2xz)/2}.$  From this it is easy to deduce that if $u\in\mathbb R^n$ is a unit vector and $a>0$ then the map
$E_{a}(z)(x)=g_a(x) e^{-za^2(x,u)^2}$ is analytic for $Re z>-a^2$.  Similarly $D_{a,u}(z)(x)=g_a(x)e^{-za^2(x,u)}$ is analytic on $\mathbb C.$

By assumption $\langle E_a(z),\rho\rangle=0$ if $z>-a^2$ is real.  Hence $\langle E_a(z),\rho\rangle=0$ for all $z$ with $\Re z>-a^2.$  In particular
$\langle E_a^{(k)}(0),\rho\rangle=0$ for $k=0,1,\ldots.$  This implies that $\langle h_u^{(2k)}g_a,\rho\rangle=0$ for all $k.$

Now $D_{a,u}^{(k)}(0)(x)= h_u^k g_a.$  Hence it follows that all the derivatives of $\rho\circ D_{a,u}(z)+\rho\circ D_{a,-u}(z)$ vanish at $0$ and thus $\langle D_{a,u}(z)+D_{a,-u}(z),\rho\rangle=0$ for all $z\in\mathbb C.$  In particular $$e^{t^2} \langle D_{a,u}(z)+D_{a,-u}(z),\rho\rangle=0, \qquad t\ge 0$$ which implies
$$ \langle \tau_{tu}g_a+\tau_{-tu}g_a,\rho\rangle=0, \qquad 0\le t<\infty.$$
Thus
$$ \langle \tau_yg_a+\tau_{-y}g_a,\rho\rangle=0, \qquad y\in \mathbb R^n.$$
\end{proof}

\begin{prop}\label{equiv} Suppose $p<0.$ Let $X$ be a normed space of dimension $n>-p.$  Then $X\in\mathcal I_p$ if and only if there is a Polish space $\Omega$, a $\sigma$-finite Borel measure $\mu$ on $\Omega$ and a linear map $T:X\to\mathcal M(\Omega,\mu)$ which is a $c$-Gaussian embedding for some $c>0.$\end{prop}

\begin{proof} First we assume that $X\in\mathcal I_p.$  Identify $X$ with $\mathbb R^n$ and suppose $\mu$ is the finite Borel measure on $S^{n-1}$ given by \eqref{Def:p<0}.  Then Lemma 3.2 of \cite{KaltonKoldobsky2005} gives that the canonical map $Tx(u)=(x,u)$ defines a $c$-Gaussian embedding of $X$ into $ L_p(S^{n-1},\mu).$

Let us prove the converse.  Assume $T:X\to\mathcal M(\Omega,\mu)$ is a $c$-Gaussian embedding of $X$ into $L_p(\Omega,\mu).$  As usual we identify $X$ with $\mathbb R^n$ and denote by $|\cdot|$ the usual Euclidean norm. Let $\{e_1,\ldots,e_n\}$ be the canonical basis. Define $\Phi:\Omega\to \mathbb R^n$ by $\Phi(\omega)=(Te_j(\omega))_{j=1}^n.$  Note that $|\Phi(\omega)|>0$ $\mu$-almost everywhere.  Let $d\mu'=|\Phi(\omega)|^pd\mu$; then $\mu'$ is a finite Borel measure on $\Omega.$  Let $\pi$ be the canonical retraction of $\mathbb R^n\setminus\{0\}$ onto $S^{n-1}$ defined by $\pi(x)=x/|x|$. We define a finite positive Borel measure $\nu$ on $S^{n-1}$ by
$\nu= c^{-p}2^{\frac{p}{2}+1}(\Gamma(-p/2))^{-1}\mu'\circ\Phi^{-1}\circ\pi^{-1}.$

Suppose $x_1,\ldots,x_n$ are linearly independent in $X$ and let $\xi=\sum_{j=1}^n\gamma_jx_j$ be an $X$-valued Gaussian process.  Let $\psi$ be the probability density function associated to this process.
Then
\begin{align}\label{scalar} \int_{\mathbb R^n} \|x\|^p\psi(x)dx&=\mathbb E\|\sum_{j=1}^n\gamma_jx_j\|^p 
=\frac{1}{c^{p}}\int_{\Omega}(\sum_{j=1}^n|Tx_j|^2)^{p/2}d\mu \nonumber \\
\text{Use the definition of the mea}&\text{sure $\mu'$ and then of $\nu.$ So the latter is equal to} \nonumber\\
&= \frac{1}{c^{p}}\int_{\Omega}(\sum_{j=1}^n(x_j,\pi\Phi(\omega))^2)^{p/2}d\mu'(\omega) \nonumber \\
&=2^{-\frac{p}{2}-1}\Gamma(p/2)\int_{S^{n-1}} (\sum_{j=1}^n(x_j,u)^2)^{p/2}d\nu(u)\\
\text{Now, by the definition of the}&\text{ Gamma function \eqref{scalar} becomes} \nonumber\\
&= \int_{S^{n-1}}\int_0^{\infty}t^{-p-1}e^{-t^2\sum_{j=1}^n(x_j,u)^2/2}dt\,d\nu(u)\nonumber\\
&= \int_{S^{n-1}}\int_0^{\infty}t^{-p-1}\hat\psi(tu)\,dt\,d\nu(u),\nonumber\end{align}
where $\hat{\psi}$ is the characteristic function of the process.

Thus if $P$ is a positive definite matrix and $\psi(x)=e^{-(Px,x)}$ then
$$ \int_{\mathbb R^n}\|x\|^p\psi(x)dx= \int_0^{\infty}t^{-p-1}\int_{S^{n-1}}\hat\psi(tu)\,d\nu(u)\,dt.$$

Let us define a distribution $\rho\in\mathcal S'$ by
$$ \langle \rho,\psi\rangle= \int_{\mathbb R^n}\|x\|^p\psi(x)dx-\int_0^{\infty}t^{-p-1}\int_{S^{n-1}}\hat\psi(tu)\,d\nu(u)\,dt.$$
Then $\rho$ satisfies the conditions of the preceding lemma, and so we have:
\begin{equation}\label{fromlem} \int_{\mathbb R^n}\|x\|^p (g_a(x+y)+g_a(x-y))dx =2\int_0^{\infty}t^{-p-1}\int_{S^{n-1}}\cos(y,tu)\hat g_a(tu)\,d\nu(u)\,dt.\end{equation}

Now let $\phi$ be an even test function on $\R^n.$ Then
$$ \phi*g_a(x)=\int_{\mathbb R^n}g_a(x-y)\phi(y)dy=\frac12\int_{\mathbb R^n}\phi(y)(g_a(x-y)+g_a(x+y))dy.$$

Thus, using the above equality, equation \eqref{fromlem} and since $\hat g_a(x)=e^{-a^2|x|^2/2},$ we have
\begin{align*} \int_{\mathbb R^n}\|x\|^p\, \phi &*g_a(x)\,dx=\\ &=\frac12 \int_{\mathbb R^n}\phi(y)\int_{\mathbb R^n}\|x\|^p(g_a(x-y)+g_a(x+y))dx\,dy\\
&=\int_{\mathbb R^n}\phi(y)\int_0^{\infty}t^{-p-1}\int_{S^{n-1}}\cos (y,tu)e^{-t^2a^{2}/2}d\nu(u)\,dt\,dy\\
\text{We apply Fubini's}&\text{ theorem to get} \\
&=\int_0^{\infty}t^{-p-1}\int_{S^{n-1}}e^{-t^2a^{2}/2}\int_{\mathbb R^n}\cos(y,tu)\phi(y)\,dy\,d\nu(u)\,dt\\
&=  \int_0^{\infty}t^{-p-1}\int_{S^{n-1}}e^{-t^2a^{2}/2}\hat\phi(tu)\,d\nu(u)\,dt,\end{align*}
since $\phi$ is even. Letting $a\to 0$ we get \eqref{Def:p<0}.\end{proof}

Let us remark that in the above Proposition the space $X$ need not be a Banach space. In other words, the existence of a Gaussian embedding of $X$ into some $L_p$ for $p<0,$ requires no convexity for its unit ball.

We will not need to consider the case $p=0$ separately; this can always be handled by reducing to the case $p<0.$  We refer the reader to \cite{KaltonKoldobskyYaskinYaskina2007} for a discussion of this case.

The following fact is very elementary but will be used repeatedly.

\begin{prop}\label{closed}  Let $X$ be a finite-dimensional normed space.  Then the set of $p$ so that $X\in\mathcal I_p$ is closed.\end{prop}

\begin{proof}  Suppose $q$ is a limit point of the set $\mathcal P=\{p:\ X\in\mathcal I_p\}.$ If $q\le -\dim X$ then the result holds trivially by the definition of $\mathcal I_p.$  Suppose $-\dim X<q<0;$ then $q\in\mathcal I_p$ by Lemma 1 of \cite{Koldobsky1998a}. For $q=0$ a modification of Theorem 6.4 of \cite{KaltonKoldobskyYaskinYaskina2007} gives the result.  If $q>0$ then the fact that $q\in\mathcal I_p$ is well-known (and
follows from considerations of positive definite functions).\end{proof}

\section{Moment functions}

In this section we will discuss moment functions of positive measurable functions on a measure space $(\Omega,\mu)$ and of random variables.

We first record for future use:

\begin{prop}\label{criterion}  Let $(\Omega,\mu)$ be a $\sigma-$finite measure space and suppose $\mathcal U$ is an open subset of $\mathbb C^n.$  Let $\phi:\Omega\times \mathbb C^n\to \mathbb C$ be a function such that for each $(z_1,\ldots,z_n)\in \mathcal U$ the map $\omega\mapsto \phi(\omega,z_1,\ldots,z_n)$ is measurable, and for each $\omega\in \Omega$ the map $(z_1,\ldots,z_n)\mapsto \phi(\omega,z_1,\ldots,z_n)$ is holomorphic on $\mathcal U.$
Let $$\Phi(z_1,\ldots,z_n)=\int_{\Omega}|\phi(\omega,z_1,\ldots,z_n)|d\mu(\omega), \qquad (z_1,\ldots,z_n)\in\mathcal U.$$  Assume that for every compact subset $K$ of $\mathcal U$ we have
$$ \sup\{\Phi(z_1,\ldots,z_n):\ (z_1,\ldots,z_n)\in K\}<\infty.$$  Then
$$ F(z_1,\ldots,z_n)=\int_{\Omega} \phi(\omega,z_1,\ldots,z_n)d\mu(\omega)$$ defines a holomorphic function on $\mathcal U.$\end{prop}

Let us assume for the moment, merely that $\mu$ is $\sigma-$finite.  The {\it distribution} of $f\in\mathcal M(\Omega,\mu)$ is the positive Borel measure $\nu_f$ on $\mathbb R$ defined by $\nu_f(B)=\mu\{\omega:\ f(\omega)\in B\}.$  If $f\in \mathcal M(\Omega,\mu)$ and $f'\in\mathcal M(\Omega',\mu')$ we write $f\approx f'$ if $f$ and $f'$ have the same distribution, i.e. $\nu_f=\nu_{f'}.$
We also write $f\otimes f'$ for the function $f\otimes f'(\omega,\omega')=f(\omega)f'(\omega')$ in $\mathcal M(\Omega\times\Omega',\mu\times\mu').$

We say that $f\in\mathcal M(\Omega,\mu)$ is {\it positive} if $\mu\{f\le 0\}=0.$  In this case $\nu_f$ restricts to a Borel measure on $(0,\infty),$ and we write $f\in\mathcal M_+(\Omega,\mu).$

\begin{prop}\label{extend}  Let $f\in\mathcal M_+(\Omega,\mu)$,  and suppose $f^p$ is integrable for $a<p<b$.  Define
$$F(z)=\int_{\O}f^z\, d\mu, \eqno a<\Re z <b.$$
Then, $F$ is analytic on the strip $a<\Re z<b$ and \newline
(i) If $\liminf_{p\to b}F(p)<\infty$ then $$\lim_{p\to b}F(p)=\int f^bd\mu<\infty.$$
\newline
(ii) If $\liminf_{p\to a}F(p)<\infty$ then $$\lim_{p\to a}F(p)=\int f^ad\mu<\infty.$$
\newline
(iii) If $F$ can be extended to an analytic function on $(\alpha,\beta)$ where $\alpha\le a<b\le \beta$ then
$f^p$ is integrable for $\alpha<p<\beta$ and
$$F(z)=\int_{\O} f^z\, d\mu, \eqno \alpha<\Re z <\beta.$$
\end{prop}

\begin{proof} The fact $F$ is analytic follows from Proposition \ref{criterion}. (i) and (ii) follow easily from Fatou's Lemma.

We now prove (iii).
Let $c$ be the supremum of all $a <\xi <b$ such that $f^z$ is integrable on $(a,\xi)$ and
$$F(z)=\int_{\O}f^z d\mu, \eqno a<\Re z <\xi.$$

\noindent
We will show that $c=\b.$ Then a similar argument for the left-hand side of the interval will complete the proof.

Assume that $ \ c<\beta.$ The function $f^z\chi _{\{f\leq 1\}}$ is integrable for $a<\Re z <\b.$ Let
$$F_0(z)=\int_{\{f\leq 1\}} f^z d\mu, \eqno \a<\Re z <\b.$$
Let $F_1(z)=F(z)-F_0(z).$ Then
$$\int_{\{f> 1\}} f^z (\log f)^m d\mu=F_1^{(m)}(z), \eqno \a<\Re z <c, \    m=0,1,2,\ldots.$$
Using Fatou's Lemma, as $z\rightarrow c,$ we see that
$$\int_{\{f> 1\}} f^c (\log f)^m d\mu\leq \liminf \int_{\{f> 1\}} f^z (\log f)^m d\mu=F_1^{(m)}(c)$$
and hence there exists $0<\tau<\beta-c$ so that
$$\int_{\{f> 1\}} f^{c+t}d\mu=\sum\limits_{m=0}^{\infty}\frac{1}{m!}\int_{\{f> 1\}} f^c (\log f)^m t^m d\mu<\infty, \eqno 0<t<\tau.$$
It follows that
$$F_1(z)=\int_{\{f> 1\}} f^z d\mu, \eqno a<\Re z<c+\tau,$$
which implies that
$$F(z)=\int_{\O}f^z d\mu, \eqno a<\Re z <c+\tau.$$
The latter contradicts the choice of $c.$
\end{proof}

We now recall the definitions and properties of some elementary random variables.  Let $\gamma$ be a normalized Gaussian random variable.  Then $\gamma$ has the distribution of the function $f(t)=t$ on $\mathbb R$ with the measure $(2\pi)^{-1/2}e^{-x^2/2}.$  We will use $(\gamma_k)_{k=1}^{\infty}$ to denote a sequence of independent normalized Gaussians defined on some probability space.

It is known that if $\gamma$ is a normalized Gaussian r.v. then for $-1<p<\infty,$ $\mathbb E(|\gamma|^p)<\infty$ . We define
\begin{equation}\label{Gz} G(z)=\mathbb E(|\gamma|^z),\qquad -1<\Re z<\infty.\end{equation}

It is in fact easy to give formulae for $G$,
\begin{equation}\label{GzG} G(z)=\frac{1}{\sqrt{\pi}}2^{z/2}\Gamma((z+1)/2)=2^{-z/2}\frac{2\Gamma(z)}{\Gamma(z/2)},\qquad -1<\Re z<\infty\end{equation}
This uses the following important formula (see \cite{Remmert1998} p.45) \begin{equation}\label{99} \Gamma(z)=\frac{2^{z-1}}{\sqrt{\pi}}\Gamma(z/2)\Gamma((z+1)/2), \qquad z\neq 0,-1,-2,\ldots.\end{equation}
It will be convenient to use $G$ in later calculations.

We denote by $\varphi_p$ a normalized positive $p$-stable random variable where $0<p<1$, which is characterized by
$$ \mathbb E(e^{-t\varphi_p})=e^{-t^p}, \eqno 0<t<\infty.$$
From the formula
$$ x^z\Gamma(-z) =\int_0^{\infty}t^{z-1}e^{-xt}\,dt$$ and analytic continuation it is easy to deduce that
\begin{equation}\label{phip} \Phi_p(z):=\mathbb E(\varphi_p^z)= \frac{\Gamma(-z/p)}{p\Gamma(-z)}, \qquad -\infty<\Re z<p.\end{equation}

Finally, for $0<p<2$ we use $\psi_p$ to denote a normalized symmetric $p$-stable random variable which is characterized by
$$  \mathbb E(e^{it\psi_p})=e^{-|t|^p}, \eqno -\infty<t<\infty.$$
It may be shown that $\psi_p\approx \sqrt{2\varphi_{p/2}}\otimes \gamma$ so that
$$ \Psi_p(z)=\mathbb E(|\psi_p|^{z})=2^{z/2} \Phi_{p/2}(z/2)G(z), \eqno -1<\Re z< p.$$

Let us remark at this point that the functions $G,\Phi_p$ and $\Psi_p$ are superfluous in that they can each be expressed fairly easily in terms of the Gamma function.  However it seems to us useful to keep them separate in order to follow some of the calculations later in the paper.

We will need the following lemma later:

\begin{lm}\label{1}  Let $\gamma_1,\ldots,\gamma_m$ be independent normalized Gaussian random variables, then if $\Re w>-1, \Re(w+z)>-m$
$$ \mathbb E|\gamma_1|^w(\gamma_1^2+\cdots+\gamma_m^2)^{z/2}= \frac{G(w)G(w+z+m-1)}{G(w+m-1)}.$$
\end{lm}

\begin{proof} It is easy to calculate
$$ \mathbb E(\gamma_1^2+\cdots+\gamma_m^2)^{z/2}= \frac{G(z+m-1)}{G(m-1)}, \qquad \Re z>-m.$$
 Note that $\gamma_1( \gamma_1^2+\cdots+\gamma_m^2)^{-1/2}$ and $(\gamma_1^2+\cdots+\gamma_m^2)^{1/2}$ are independent.
Hence for $\Re w>-1$ $$G(w) =\mathbb E(|\gamma_1|^w)= \mathbb E(|\gamma_1|^w(\gamma_1^2+\cdots+\gamma_m^2)^{-w/2})\frac{G(w+m-1)}{G(m-1)}.$$
Thus
$$\mathbb E(|\gamma_1|^w(\gamma_1^2+\cdots+\gamma_m^2)^{-w/2})=\frac{G(w)G(m-1)}{G(w+m-1)}.$$
Finally, again using independence
$$\mathbb E|\gamma_1|^w(\gamma_1^2+\cdots+\gamma_m^2)^{z/2})=\frac{G(w)G(w+z+m-1)}{G(w+m-1)}.$$\end{proof}

\section{Mellin transforms and absolute norms}

Let $f$ be a complex-valued Borel function on $(0,\infty)$. Let $J_f$ be the set of $a\in\mathbb R$ such that
$$ \int_0^{\infty}t^{-a}|f(t)|\frac{dt}t<\infty.$$ It is known that $J_f$ is an interval (possibly unbounded) which may be degenerate (a single point) or empty.  If $J_f\neq \emptyset$ we define the {\it Mellin transform,} of $f$ by
$$ Mf(z)=\int_0^{\infty} t^{-1-z}f(t)\,dt, \eqno z\in J_f.$$ Then by Proposition \ref{criterion}, $Mf$ is analytic on the interior of $J_f$ (if this is nonempty). For the general theory of the Mellin transform we refer to \cite{Zemanian1987}.

The following are some basic facts about the Mellin transform that will be used throughout this article. The first part of the Proposition is a Uniqueness theorem of the transformation.

\begin{prop}\label{Mellin} (i) Suppose $f,g$ are two Borel functions defined on $(0,\infty)$ and $a\in J_f\cap J_g.$  If $Mf(a+it)=Mg(a+it)$ for $-\infty<t<\infty$ then $f(t)=g(t)$ almost everywhere.

(ii) Suppose $f$ is a Borel function on $(0,\infty)$ Suppose $E$ is an analytic function on the strip $a<\Re z<b$ and that there exist $a\le c<d\le b$ so that $(c,d)\subset J_f$ and $Mf(z)=E(z)$ for $c<\Re z<d.$  Then $(a,b)\subset J_f$ and $Mf(z)=E(z)$ for $a<\Re z<b$. \end{prop}

\begin{proof} For (i) see \cite{Zemanian1987}, Theorem 4.3-4, while (ii) is a restatement of Lemma \ref{extension} for the measure $f(t)dt/t.$\end{proof}

Let $N$ be a normalized absolute norm on $\mathbb R^2.$  Thus $N$ is a norm satisfying $N(0,1)=N(1,0)=1$ and $N(u,v)\le N(s,t)$ whenever $|u|\le |s|$ and $|v|\le |t|.$
We define an analytic function of two variables by
$$ F_N(w,z)=\int_0^{\infty}t^{-z-1}N(1,t)^{w+z}dt,\eqno \Re z<0,\ \Re w<0.$$
For $p<0$ the Mellin transform of $N(1,t)^p$ is given
by $$M_{p,N}(z)=F_N(p-z,z), \eqno \Re z<0.$$
Notice that if $N'(s,t)=N(t,s)$ then $F_{N'}(w,z)=F_N(z,w).$  Thus $M_{p,N'}(z)=M_{p,N}(p-z)$ for $\Re z<0.$

For the special case of the $\ell_\infty-$norm we define
\begin{equation}\label{Finf}F_\infty(w,z)=\int_0^{\infty}t^{-z-1}\max\{1,t\}^{w+z}dt=-\frac1z-\frac1w, \qquad \Re z<0,\ \Re w<0.\end{equation}
 We write \begin{equation}\label{Minf}M_{p,\infty}(z)=\frac{p}{z(z-p)}, \qquad \Re z<0.\end{equation}

The following lemma is an immediate deduction from the Mean Value Theorem:

\begin{lm}\label{estimates}  Suppose $w\in \mathbb C.$  Then:
\begin{equation}\label{firstestimate} |(1+t)^w-1|\le |w|2^{\Re w-1}t\le 2^{2|w|}t, \qquad 0\le t\le 1\end{equation}
and
\begin{equation}\label{secondestimate}|\tfrac12((1+t)^w+(1-t)^w)-1|\le |w|(|w|+1)2^{\Re w-2}t^2\le 2^{3|w|}t^2,\qquad 0\le t\le 1/2.\end{equation}
\end{lm}

In view to Lemma \ref{estimates} we define $$\tilde F_N(w,z)=\int_0^{\infty}t^{-z-1}(N(1,t)^{w+z}-\max\{1,t\}^{w+z})\,dt$$ on the region
$\{(w,z):\ \Re w<1,\ \Re z<1\}.$  Then applying analytic continuation we have  $$\tilde F_N(w,z)=F_N(w,z)+\frac1w+\frac1z.$$

The following lemma is immediate, using Proposition \ref{criterion} and equations \eqref{firstestimate}, \eqref{secondestimate}:

\begin{lm}\label{estimates2}  Suppose $1\le r,s<\infty$ and $N$ is a normalized absolute norm satisfying the estimates
$$ N(1,t)^r\le 1+Ct^r,\eqno 0\le t\le 1$$ and
$$ N(t,1)^s\le 1+Ct^s,\eqno 0\le t\le 1.$$
Then $\tilde F_N$ extends to an analytic function of $(w,z)$ on the region $S=\{(w,z):\ \Re w<s,\ \Re z<r\}.$
\end{lm}

This Lemma allows us to define $F_N(w,z)$ when $\Re w<s,\ \Re w<r$ and $w,z\neq 0.$  We may then extend the definition of $M_{p,N}(z)$ to the case $p<r$ and $0<\Re z<p$; then $M_{p,N}$ is an analytic function on this strip.

The following proposition explains our interest in the function $F_N.$

\begin{prop}\label{absolute}  Let $X$ and $Y$ be two normed spaces and let $Z=X\oplus_NY$.  If $x\in X\subset Z$ and $y\in Y\subset Z$ with $\|x\|,\|y\|\neq 0$ then
\begin{equation}\label{A1} \int_0^{\infty}t^{-z-1}\|x+ty\|^{w+z}dt = F_N(w,z)\|x\|^w\|y\|^z,\qquad \Re w,\Re z<0.\end{equation}
\end{prop}

\begin{proof} Assuming $\|x\|,\|y\|\neq 0$, we observe that
\begin{align*}
\int_0^{\infty}t^{-z-1}\|x+ty\|^{w+z}dt&= \|x\|^{w+z}\int_0^{\infty}t^{-1-z}N(1,t\|y\|/\|x\|)^{w+z}dt\\
&= \|x\|^w\|y\|^z \int_0^{\infty}t^{-1-z}N(1,t)^{w+z}dt.\end{align*}
\end{proof}

Let us recall the Euler Beta function:
$$ B(w,z) = \int_0^1 x^{w-1}(1-x)^{z-1}dx = \frac{\Gamma(w)\Gamma(z)}{\Gamma(w+z)},\qquad \Re w,\Re z>0.$$
Making the substitution $x=(1+t)^{-1}$ we get the alternative formula:
\begin{equation}\label{101} B(-w,-z) =\int_0^{\infty} t^{-z-1}(1+t)^{w+z}dt,\qquad \Re w,\Re z<0.\end{equation}
Hence if $u,v>0$ and $\Re w,\ \Re z<0,$ we have
\begin{equation}\label{102}\int_0^{\infty} t^{-z-1}(\a^p+\b^pt^p)^{(w+z)/p}dt=\frac1p\a^{w}\b^{z}B(-w/p,-z/p).\end{equation}
In particular if $N(s,t)=(|s|^q+|t|^q)^{1/q}$ is the $\ell_q-$norm we have an explicit formula for $F_q=F_N$
\begin{equation}\label{Fq} F_q(w,z)=\frac1qB(-w/q,-z/q), \qquad \Re w,\Re z<0.\end{equation}
As before we regard \eqref{Fq} as the definition of $F_q$ when $\Re w<q,\ \Re z<q$ and $w,z\neq 0.$
Then for $p<0$ we can define
\begin{equation}\label{Mpq}M_{p,q}(z) =\frac1qB((z-p)/q,-z/q), \qquad \Re z<0.\end{equation}  If $0<p<q$ the same definition gives an analytic function on $0<\Re z<p.$

\begin{lm}\label{extension}   Suppose $1\le r,s<\infty$ and $N$ is a normalized absolute norm satisfying the estimates
$$ N(1,t)^r\le 1+Ct^r,\eqno 0\le t\le 1$$ and
$$ N(t,1)^s\le 1+C't^s,\eqno 0\le t\le 1.$$  Then the function $(w,z)\mapsto F_N(w,z)/F_2(w,z)$ extends to a holomorphic function on the region $\{(w,z):\ \Re w<\min\{s,2\},\ \Re z <\min\{r,2\}\}.$

Thus for $p<0$, $z\mapsto M_{p,N}(z)/M_{p,2}(z)$ extends to an analytic function on the strip $\{z:\ p-\min\{s,2\}<\Re z<\min\{r,2\}\}.$\end{lm}

\begin{proof} This follows directly from the definition of $F_2$ and Lemma \ref{estimates}.\end{proof}

\begin{lm}\label{symmetrize} For $\Re z,\ \Re w<0$ and $\Re(w+z)>-1$  we have
$$ \frac12\int_0^{\infty}t^{-z-1}(|1+t|^{w+z}+|1-t|^{w+z})\,dt = \frac{G(w+z)F_2(w,z)}{G(w)G(z)}.$$
\end{lm}

\begin{proof} Let $$ Q(w,z)=\frac12\int_0^{\infty}t^{-z-1}(|1+t|^{w+z}+|1-t|^{w+z})\,dt.$$ Let $\gamma_1,\gamma_2$ be two normalized independent Gaussian random variables on some probability space.
Then by \eqref{Gz}
$$ \mathbb E(|\gamma_1+t\gamma_2|^{w+z})= (1+t^2)^{\frac{(w+z)}{2}} G(w+z)\eqno \Re (w+z)>-1.$$
Hence using \eqref{101} and \eqref{Fq} we have
$$ \int_0^\infty t^{-z-1}\mathbb E(|\gamma_1+t\gamma_2|^{w+z})\,dt= G(w+z)F_2(w,z).$$

Note that the function $t^{-z-1}|\gamma_1+t\gamma_2|^{w+z}$ is integrable on the product space as long as $\Re z,\Re w<0$ and $\Re (w+z)>-1.$  Thus we can apply Fubini's theorem and a change of variables $t|\gamma_2|=s|\gamma_1|$ to obtain
\begin{align*}
G(w+z)&F_2(w,z)= \mathbb E\left(\int_0^{\infty} t^{-z-1}|\gamma_1+t\gamma_2|^{w+z}\,dt\right)\\
&= \frac12\mathbb  E\left(\int_0^{\infty}t^{-z-1}\left(\Big||\gamma_1|+t|\gamma_2|\Big|^{w+z}+\Big||\gamma_1|-t|\gamma_2|\Big|^{w+z}\right)dt\right)\\
&=\mathbb E\left(Q(w,z)|\gamma_1|^{w}|\gamma_2|^{z}\right).\end{align*}
Then using \eqref{Gz} the Lemma follows.\end{proof}

\begin{lm}\label{independent} Let $(\Omega,\mu)$ be a $\sigma-$finite measure space  and suppose $f,g\in \mathcal M(\Omega,\mu).$  Then if $w,z \in \C$ are such that $\Re w,\Re z<0$ and $$\int_{\Omega}|f|^{\Re w}|g|^{\Re z}d\mu<\infty$$ we have
\begin{equation}\label{independent0} \int_0^{\infty}t^{-z-1}\int_{\Omega} (f^2+t^2g^2)^{w+z}d\mu = F_2(w,z) \int_{\Omega}|f|^w|g|^zd\mu.\end{equation}
Further, if $\Re(w+z)>-1$ we have
\begin{equation}\label{independent1} \int_0^{\infty}t^{-z-1}\int_{\Omega} (|f+tg|^{w+z}+|f-g|^{w+z})d\mu\,dt =\frac{G(w+z)}{G(w)G(z)} F_2(w,z) \int_{\Omega}|f|^w|g|^zd\mu.\end{equation}
\end{lm}

\begin{proof}  We first use Tonelli's theorem for $u=\Re w$ and $v=\Re z.$ Then by \eqref{102} and \eqref{Fq} we have
$$\int_0^{\infty}t^{-1-v}\int_{\Omega} (f^2+t^2g^2)^{u+v}d\mu = F_2(u,v) \int_{\Omega}|f|^u|g|^vd\mu,$$ where both integrals converge.  Then applying Fubini's theorem we get \eqref{independent0}.  The proof of \eqref{independent1} is precisely similar using Lemma \ref{symmetrize}.\end{proof}

Suppose $(\Omega,\mu)$ is a probability space and $h$ is  a symmetric function in $L_p(\Omega,\mu)$ where $p>0.$  In the following Lemmas we show how to compute the Mellin transform of the function $t\mapsto \|1+th\|_p-\max\{1,t\}^p.$

\begin{lm}\label{later} Let $(\Omega,\mu)$ be a probability
space and suppose $h\in\mathcal M(\Omega,\mu).$  Suppose $-1<a<0<b<2$ and that
$$ \int_{\Omega}(|h|^a+|h|^b)d\mu<\infty.$$   Let
$$H(z)=\int_{\Omega}|h|^z\,d\mu, \qquad a<\Re z<b.$$ Then
$$E(w,z)=\int_0^{\infty}t^{-z-1} \int_{\Omega}\frac12(|1+th|^{w+z}+|1-th|^{w+z}-2\max\{1,t|h|\}^{w+z})d\mu\,dt$$ defines a holomorphic function on the
region $\mathcal U=\{(w,z):\ a<\Re(w+z),\ -1< \Re w<2,\ a<\Re z<b\}$ and
\begin{equation}\label{later0} E(w,z)=\left(\frac{G(w+z)F_2(w,z)}{G(w)G(z)}+\frac1w+\frac1z\right)H(z),\end{equation}
when $(w,z)\in\mathcal U, \Re w>-1,\ w,z\neq 0.$\end{lm}

\begin{proof} For $t>0$ and $w,z\in\mathbb C,$ we consider
$$ \varphi(t,w,z)=t^{-z-1}(|1+t|^{w+z}+|1-t|^{w+z}-2\max\{1,t\}^{w+z}).$$
Let $u=\Re w,\ v=\Re z.$  Then by Lemma \ref{estimates} we have
$$ |\varphi(t,w,z)| \le 2^{3|w+z|}t^{1-v}, \qquad 0\le t\le 1/2$$ and
$$ |\varphi(t,w,z)|\le 2^{3|w+z|}t^{u-3}, \qquad 2\le t<\infty.$$
For $1/2\le t\le 2$ we have the estimates
$$ |\varphi(t,w,z)|\le 2^{1+|v|}2^{u+v+2}, \qquad u+v\ge 0$$ and
$$ |\varphi(t,w,z)|\le 2^{3+|v|}|1-t|^{u+v}, \qquad u+v<0.$$

Thus if $v<2,\ u<2, \ u+v>-1$ we have a very crude estimate:
$$ \int_0^{\infty}|\varphi(t,w,z)|dt \le 2^{3|w+z|}\left(\frac{1}{2-u}+\frac{1}{2-v}\right)+ 2^{4+|v|}\frac{1}{u+v+1}.$$

Now
$$ t^{-z-1}(|1+th|^{w+z}+|1-th|^{w+z}-2\max\{1,t|h|\}^{w+z})=|h|^{1+z} \varphi(t|h|,w,z)$$ and so
\begin{align*} &\int_0^{\infty}\int_{\Omega}\left|t^{-z-1}(|1+th|^{w+z}+|1-th|^{w+z}-2\max\{1,t|h|\}^{w+z})\right|d\mu\,dt \\
&=H(v)\int_0^{\infty}|\varphi(t,w,z)|dt.\end{align*}

Combining these estimates shows that we have the conditions of Proposition \ref{criterion} for the region $\mathcal U$ and so $E$ defines a holomorphic function on $\mathcal U.$

For $(w,z)\in\mathcal U$ and $\Re w,\Re z<0$ we can use Lemma \ref{symmetrize} to show that
$$\int_0^{\infty}t^{-z-1} \int_{\Omega}\frac12(|1+th|^{w+z}+|1-th|^{w+z})d\mu\,dt =\frac{G(w+z)F_2(w,z)}{G(w)G(z)}\int_{\Omega}|h|^zd\mu$$ and
$$ \int_0^{\infty}t^{-z-1} \int_{\Omega}\max\{1,t|h|\}^{w+z}t =(-\frac1w-\frac1z)\int_{\Omega}|h|^zd\mu.$$  Since the right-hand side of \eqref{later0} extends to an analytic function in $\mathcal U$, \eqref{later0} holds for all $(w,z)\in\mathcal U$ with $w,z\neq 0$.
\end{proof}

\begin{lm}\label{anotherone} Let $(\Omega,\mu)$ be a probability space and suppose $h\in\mathcal M(\Omega,\mu).$  Suppose $-1<a<0<b<2$ and that
$$ \int_{\Omega}(|h|^a+|h|^b)d\mu<\infty.$$  Then
$$ E_0(w,z)=\int_0^{\infty}t^{-z-1}\int_{\Omega}(\max\{1,th\}^{w+z}-\max\{1,t\}^{w+z})d\mu\,dt$$ defines an analytic function on the region $\mathcal U_0=\{(w,z):\ a<\Re (w+z),\ \Re w<0,\ \Re z<b\}.$
Furthermore
\begin{equation}\label{anotherone0} E_0(w,z)= (\frac1w+\frac1z)(1-H(z)), \qquad (w,z)\in\mathcal U_0.\end{equation}\end{lm}

\begin{proof}  Let $u=\Re w$ and $v=\Re z.$
Then if $s>0$ we have
\begin{align*} \int_0^{\infty}t^{-1-v}|\max\{1,st\}^{w+z}-\max\{1,t\}^{w+z}|dt&\le  \int_{1/s}^\infty s^{u+v}t^{-1+u}dt +\int_1^{\infty}t^{-1+u}dt\\
&\le \frac{s^v+1}{|u|}.\end{align*}
Hence
$$ \int_0^{\infty}t^{-1-v}\int_{\Omega}|\max\{1,th\}^{w+z}-\max\{1,t\}^{w+z}|d\mu\,dt \le \frac1{|u|}(H(v)+1).$$
Again Proposition \ref{criterion} gives that $E_0$ defines a holomorphic function on $\mathcal U_0.$

If in addition $\Re z<0$ we can compute
$$\int_0^{\infty}t^{-z-1}\int_{\Omega}\max\{1,th\}^{w+z}d\mu\,dt = -H(z)(\frac1w+\frac1z)$$ and
$$\int_0^{\infty}t^{-z-1}\int_{\Omega}\max\{1,t\}^{w+z}d\mu\,dt = -(\frac1w+\frac1z).$$
As before analytic continuation gives \eqref{anotherone0} throughout $\mathcal U_0.$\end{proof}

Combining the preceding Lemmas we have the following:

\begin{prop}\label{Mellinh} Let $(\Omega,\mu)$ be a probability
space and suppose $h\in\mathcal M(\Omega,\mu)$ is a symmetric random variable.  Suppose $-1<a<0<b<2$ and that
$$ \int_{\Omega}(|h|^a+|h|^b)d\mu<\infty.$$   Let
$$H(z)=\int_{\Omega}|h|^z\,d\mu, \qquad a<\Re z<b.$$
Suppose $0<p<b$ is such that $H(p)=1.$  Then the Mellin transform of $t\mapsto \int_{\Omega}(|1+th|^p-\max\{1,t\}^p)\,d\mu$ is given
by
\begin{equation}\label{eqMellinh0}
\int_0^{\infty}t^{-z-1} \int_\Omega (|1+th|^p-\max\{1,t\}^p)d\mu\,dt= \frac{G(p)M_{p,2}(z)H(z)}{G(p-z)G(z)}+\frac{p}{z(p-z)},\end{equation} for $
a<\Re z<\min\{b,p+1\}.$
\end{prop}

Let us remark that, since $H(0)=H(p)=1,$ the right-hand side of \eqref{eqMellinh0} has removable singularities at $z=0$ and $z=p.$

\begin{proof} Since the right-hand side is analytic in the strip $a<\Re z<\min\{b,p+1\}$ it follows from Proposition \ref{Mellin} that it is necessary only to establish equality for the strip $p<\Re z<\min\{b,p+1\}.$  In this case $-1<\Re (p-z)<0$ and so $(p-z,z)\in \mathcal U\cap\mathcal U_0$ as these sets are defined in Lemmas \ref{later} and \ref{anotherone}.  Since $h$ is symmetric we can rewrite the left-hand side of \eqref{eqMellinh0} in the form
$$\int_0^{\infty}t^{-z-1} \int_\Omega \frac12(|1+th|^p+|1-th|^p-2\max\{1,t\}^p)d\mu\,dt.$$
Then combining Lemmas \ref{later} and \ref{anotherone} we get the conclusion.\end{proof}

This Proposition can be extended by an approximation argument to the case when $a=0$ and $b=p$; we will not need this so we simply state the result:

\begin{prop}\label{Mellinhp} Let $(\Omega,\mu)$ be a probability
space and suppose $h\in  L_p(\Omega,\mu)$, with $\|h\|_p=1,$ where $0<p<2.$
  Let
$$H(z)=\int_{\Omega}|h|^z\,d\mu, \qquad a<\Re z<b.$$
Suppose $0<p<b$ is such that $H(p)=1.$  Then the Mellin transform of $t\to \int_{\Omega}|1+th|^p-\max\{1,t\}^p\,d\mu$ is given
by
$$\int_0^{\infty}t^{-z-1} \int_\Omega (|1+th|^p-\max\{1,t\}^p)d\mu\,dt= \frac{G(p)M_{p,2}(z)H(z)}{G(p-z)G(z)}+\frac{p}{z(p-z)},$$
for $0<\Re z<p.$
\end{prop}

\section{Embedding $X\oplus_NY$ into $L_p$}

\begin{prop}\label{p<0Z}
Let $X, Y$ be two non-trivial normed spaces, with $dim X =m$ and $dim Y=n.$
and suppose $N$ is a normalized absolute norm on $\R^2.$
Suppose $T:X\oplus_N Y\to \mathcal M(\Omega,\mu)$ is a 1-Gaussian embedding into $L_p(\Omega,\mu)$ where $-(n+m)<p<0.$ Suppose $x_1,\ldots,x_n \in X$ and $y_1\ldots, y_m \in Y$ are linearly independent, and suppose $\xi=\sum\limits_{j=1}^{m}\gamma_j x_j$ and $\eta=\sum\limits_{j=1}^{n}\gamma^{'}_j y_j$ are independent Gaussian processes of full rank with values in $X$ and $Y$ respectively.  Then for $\max\{-n,p\}<\Re z<\min\{0,p+m\}$ we have:
\begin{equation}\label{eq:p<0Z}
\int_{\O}\Bigl(\sum\limits_{j=1}^{m}(Tx_j)^2\Bigr)^{\frac{p-z}{2}}\Bigl(\sum\limits_{j=1}^{n}(Ty_j)^{2}\Bigr)^{\frac{z}{2}}d\mu=\frac{M_{p,N}(z)}{M_{p,2}(z)}\E \|\xi\|^{p-z} \E \|\eta\|^{z}.\end{equation}
\end{prop}

\begin{proof} By assumption we have
$$ \mathbb E\|\xi+t\eta\|^p = \int_{\Omega}\bigl(\sum_{j=1}^m(Tx_j)^2+t^2\sum_{j=1}^n(Ty_j)^2\bigr)^{p/2}d\mu,\eqno t>0.$$
Hence if $\max\{-n,p\}<\Re z<\min\{0,p+m\}$ we have
$$ \int_0^{\infty}t^{-z-1}\mathbb E\|\xi+t\eta\|^p\,dt=\int_{\Omega}\int_0^{\infty}t^{-z-1}\bigl(\sum_{j=1}^m(Tx_j)^2+t^2\sum_{j=1}^n(Ty_j)^2\bigr)^{p/2}d\mu$$ and both sides are integrable. Notice that, in particular, it follows that $\sum_{j=1}^m (Tx_j)^2>0$ and $\sum_{j=1}^n (Ty_j)^2>0,$ $\mu-$almost everywhere.

Now for real $\max\{n,-p\}<u<\min\{0,p+m\},$ using Tonelli's theorem and \eqref{A1},
 \begin{align*}\int_0^{\infty}t^{-u-1}\mathbb E\|\xi+t\eta\|^p\,dt&=\mathbb E\int_0^{\infty}t^{-u-1}N(\xi,t\eta)^p\,dt\\
 &= F_N(p-u,u) \mathbb E\|\xi\|^{p-u}\|\eta\|^u\\
 &= F_N(p-u,u) \mathbb E\|\xi\|^{p-u}\mathbb E\|\eta\|^u,\end{align*}
 since $\xi$ and $\eta$ are independent. We repeat the calculation replacing $u$ by complex $z$ and apply Fubini's theorem. Then
$$ \int_0^{\infty}t^{-z-1}\mathbb E\|\xi+t\eta\|^p\,dt= M_{p,N}(z)\mathbb E\|\xi\|^{p-z}\mathbb E\|\eta\|^z,$$
for $\max\{-n,p\}<\Re z<\min\{0,p+m\}.$
Hence (first for real $z$, using Tonelli's theorem and then for the general case), by Lemma \ref{independent} we get
\begin{align*} & M_{p,N}(z)\E \|\xi\|^{p-z} \E \|\eta\|^{z}=\\ &=M_{p,2}(z)\int_\Omega \Bigl(\sum\limits_{j=1}^{m}(Tx_j)^2\Bigr)^{\frac{p-z}{2}}\Bigl(\sum\limits_{j=1}^{n}(Ty_j)^{2}\Bigr)^{\frac{z}{2}}d\mu,\end{align*}
which proves \eqref{eq:p<0Z}.
\end{proof}

We shall say that an embedding $T:X\to \mathcal M(\Omega,\mu)$ is {\it isotropic} if $Tx\approx Tx'$ whenever $\|x\|=\|x'\|=1.$  We will say that it is $f-$isotropic if $f$ is a  Borel function on some $\sigma-$finite Polish measure space $(K,\nu)$ and $Tx\approx f$ for every $x\in X$ with $\|x\|=1.$ For $0<p<2$,
$T$ is a $p$-stable embedding if $Tx\approx \psi_p$ whenever $\|x\|=1.$  If $X$ embeds into $L_p$ then there is a $p$-stable embedding of $X$ into $\mathcal M(\Omega,\mu),$ where $\mu$ is a probability measure.

\begin{prop}\label{p>0Z}Let $X, Y$ be two normed spaces, with $dim X =m$ and $dim Y=n,$
and suppose $N$ is a normalized absolute norm on $\R^2.$
Suppose $T:X\oplus_N Y\to \mathcal M(\Omega,\mu)$ is a $p$-stable embedding where $p>0.$
Then for any nonzero $x\in X$ and $y\in Y$ and $-1<\Re(w+z)<\Re w,\Re z<0,$ we have
\begin{equation}\label{eq:p>0Z}
\int_\Omega |Tx|^{w}|Ty|^z\,d\mu = \frac{F_N(w,z)G(w)G(z)\Phi_{p/2}((w+z)/2)}{F_2(w,z)}\|x\|^w\|y\|^z.
\end{equation}\end{prop}

\begin{proof}  If $f\in X\oplus_N Y$ we have
$$ \int_\Omega  |Tf|^z \,d\mu = \Psi_p(z)\|f\|^z,\eqno \Re z>-1.$$
Now consider $\xi=\gamma_1x$ and $\eta=\gamma_2y$ where $\gamma_1,\gamma_2$ are normalized independent Gaussian random variables.
Then
$$ \mathbb E\int_{\Omega} |T\xi+tT\eta|^z\,d\mu =\Psi_p(z)\mathbb E\|\xi+t\eta\|^z.$$
If $-1<\Re(w+z)<\Re w,\Re z<0,$ then by Fubini's theorem, Proposition \ref{absolute} and \eqref{GzG} (first for real $w,z$ using Tonelli's theorem as in Proposition \ref{p<0Z}), we have that
\begin{align}\label{eq:p>0Z1} \int_0^{\infty}t^{-z-1}\mathbb E\|\xi+t\eta\|^{w+z}\,dt  &=\int_0^{\infty}t^{-z-1}\mathbb EN(\xi,t\eta)^{w+z}\,dt \nonumber \\
&= F_N(w,z)\mathbb E\|\xi\|^w\mathbb E\|\eta\|^z \nonumber \\
&= F_N(w,z) G(w)G(z) \|x\|^w\|y\|^z.\end{align}
On the other hand, $\gamma_1,\gamma_2$ are Gaussian r.v.
\begin{align}\label{eq:p>0Z2}
&\int_0^{\infty}t^{-z-1}\mathbb E\int_\Omega |T\xi+tT\eta|^{w+z}\,d\mu\,dt \nonumber \\
&=G(w+z)\int_\Omega\int_0^\infty t^{-z-1} ((Tx)^2+t^2(Ty)^2)^{z/2}dt\,d\mu \nonumber \\
\text{and by Lemma \ref{independent} the}& \text{ latter is equal to}  \nonumber \\
&= G(w+z)F_2(w,z)\int_\Omega |Tx|^w|Ty|^z\,d\mu.\end{align}
Then equation \eqref{eq:p>0Z} follows from \eqref{eq:p>0Z1}, \eqref{eq:p>0Z2} and the fact that $\Psi_p(z)=\Phi_{p/2}(z/2)G(z).$
\end{proof}

\begin{Thm}\label{main}
Let $X,Y$ be two non-trivial finite dimensional normed spaces with dimensions $m$ and $n$ respectively. Suppose that $-(n+m)<p\le 1\le r,s\leq 2$ and that $N$ is a normalized absolute norm on $\R^2$ satisfying estimates of the type
\begin{equation}\label{eq:N}
N(1,t)^r\leq 1+Ct^r, \qquad t>0,
\end{equation} and
\begin{equation}\label{eq:N}
N(t,1)^s\leq 1+C't^s, \qquad t>0.
\end{equation}
If $X\oplus_N Y\in\mathcal I_p$ then $X\in\mathcal I_q$  whenever $p-r\le q\le \min\{s,p+n\}$ and $Y\in \mathcal I_q$  whenever $p-s\leq q \leq \min \{r, p+m\}$
\end{Thm}

\begin{proof}  It suffices to consider the case of $Y$ and to prove the result if $p-s<q<\min\{r,p+m\}.$ Then the limiting case follows by Proposition \ref{closed}. We will treat the cases $p<0,\ p=0$ and $0<p\le 1$ separately.

{\it Case 1: } Let $p<0.$ The space $X\oplus_N Y$ embeds into $L_p$ so we can consider a 1-Gaussian embedding $T:X\oplus_NY\to L_p(\O,\mu).$  By Proposition \ref{p<0Z}, for any linearly independent sets $x_1,\ldots, x_m \in X$ and $y_1,\ldots, y_n \in Y$ equation \eqref{eq:p<0Z} holds  in the strip $\max \{-n,p\}<\Re z <\min \{p+m,0\}.$  However by Lemma \ref{extension} the function $M_{p,N}(z)/M_{p,2}(z)$ can be analytically continued to the strip $p-s<\Re z<r.$  Thus the right-hand side of \eqref{eq:p<0Z} can be analytically continued to the strip $\max\{p-s,-n\}<\Re z<\min \{r,p+m\}.$   By Proposition \ref{extend} this implies that \eqref{eq:p<0Z} holds (and both sides are integrable) in the strip $\max\{p-s,-n\}<\Re z<\min \{r,p+m\}.$  If $\max\{p-s,-n\}<q<\min \{r,p+m\}$ and $q\neq 0$, we fix some $\xi$ so that $\mathbb E\|\xi\|^{p-q}=1.$  Let $f=(\sum_{j=1}^m(Tx_j)^2)^{1/2}.$ Then

$$ \frac{M_{2,N}(q)}{M_{2,p}(q)}\mathbb E\|\eta\|^q = \int_\O (\sum_{j=1}^n(Ty_j)^2)^{q/2} f^{p-q}d\mu.$$ In particular $M_{2,N}(q)$ cannot vanish and $T$ is a Gaussian embedding of $Y$ into $L_q(f^{p-q}d\mu).$

If $q=0$ we note that our proof yields $Y\in\mathcal I_{\varepsilon}$ for sufficiently small $\varepsilon>0$ and so $Y\in\mathcal I_0.$

It follows that $Y\in \mathcal I_q$ for $p-s\le q\le \min (r,p+m).$ (Our convention implies $Y\in\mathcal I_q$ if $q\le -n.$)

\medskip

{\it Case 2:} Let $p=0.$ In this case $X\oplus_N Y\in \mathcal I_p$ for all $p<0$ and the result follows from Case 1.

\medskip

{\it Case 3:} Now we assume that $0<p\le 1.$  Again we prove the result for $Y$.  If $m\ge 2$ then $X\oplus_NY\in\mathcal I_0$  (\cite{KaltonKoldobskyYaskinYaskina2007} and \cite{Koldobsky1999b})and by Case 2 we have that $Y\in\mathcal I_r.$  Thus we only consider the case $m=1.$  Suppose that $X\oplus_N Y$ embeds into $\M (\O,\mu)$ via a $p$-stable embedding $T$.  We fix $x\in X$ with $\|x\|=1$ and $p<q<\min\{p+1,r\}$; let $f=|Tx|.$  Fix $a>0$ so that $q+a<p+1.$  Then $0<a<1$ and so by \eqref{eq:p>0Z} we have
\begin{equation}\label{Ty}\int_\Omega |Ty|^zf^{a-1}\,d\mu = \frac{F_N(a-1,z)G(a-1)G(z)\Phi_{p/2}((z+a-1)/2)}{F_2(a-1,z)}\|y\|^z,\end{equation}
where $ y\in Y,$ as long as
$-a<\Re z<0.$  However $F_N(a-1,z)/F_2(a-1,z)$ can be analytically continued to the half-plane $\Re z< r$ (by Lemma \ref{extension}).  We also have that $\Phi((z+a-1)/2)$ can be analytically continued to the half-plane $\Re z<p+1-a.$  Hence the right-hand side can be analytically continued to the strip $-1<\Re z< \min(r,p+1-a).$  By Lemma \ref{extend} this means that the left-hand side of \eqref{Ty} is integrable and equality holds for $-1<\Re z<\min\{r,p+1-a\}.$  In particular
$$ \int_\Omega |Ty|^q f^{a-1}\,d\mu = c\|y\|^q, \eqno y\in Y$$ where $c$ is a positive constant. This implies the result, since $Y\in\mathcal I_q$ whenever $p\le q.$
\end{proof}

The next result is known; it follows from Koldobsky's Second Derivative test (Theorem 4.19 of \cite{Koldobsky2005}; see also \cite{Koldobsky1998}).

\begin{Thm}\label{secondderivative} Let $N$ be a normalized absolute norm on $\mathbb R^2$ such that
$$ \lim_{t\to 0}\frac{N(1,t)-1}{t^2}=0.$$
Then if $-\infty<p<0$ and $X\oplus_N \mathbb R$ embeds into $L_p$ we have $\dim X\le 2-p.$
\end{Thm}

Note that the result of Theorem \ref{secondderivative} can be extended for $p\in (-\infty,2).$ Here, we present only the proof for $p<0.$

\begin{proof}  first we observe that $M_{p,N}(z)/M_{p,2}(z)$ extends to an analytic function on $-p-1<\Re z<2$ and that
\begin{equation}\label{limzero} \lim_{r\to 2} \frac{M_{p,N}(r)}{M_{p,2}(r)}=0.\end{equation}
To see \eqref{limzero} we note that by definition, for $0<r<2$ we have
$$ M_{N,p}(r) =-\frac1r-\frac1{p-r}+\int_0^{\infty}t^{-1-r}(N(1,t)^p-max\{1,t\}^p)\,dt.$$
Fix any $0<\varepsilon<1$ and let

$$\delta=\delta(\varepsilon)=\sup_{t\le\varepsilon}\frac{N(1,t)^p-1}{t^2}.$$
Then
$$ \left|M_{N,p}(r)+\frac{p}{r(p-r)} -\int_{\varepsilon}^{\infty}t^{-1-r}(N(1,t)^p-1)\,dt\right|\le \frac{\delta}{2-r}.$$
It follows that
\begin{equation}\label{lim} \limsup_{r\to 2}(2-r)M_{p,N}(r) \le \delta.\end{equation}  Since $\lim_{\varepsilon\to 0}\delta(\varepsilon)=0$ we obtain \eqref{limzero}.

Now suppose $m=\dim X>2-p$ and assume $T:X\oplus_NY \to\mathcal M(\Omega,\mu)$ is a 1-Gaussian embedding into $L_p(\Omega,\mu),$ where $\dim Y=1.$  Let us fix $\xi=\sum_{j=1}^m\gamma_jx_j$, an $X$-valued Gaussian process of full rank and $\eta=\gamma'y$ where $y\in Y$ has norm one and $\gamma'$ is a Gaussian r.v. Then, if $f=(\sum\limits_{j=1}^m(Tx_j)^2)^{1/2}$ and $g=|Ty|,$ by Proposition \ref{p<0Z}  we have
$$ \int_{\Omega}f^{p-z}g^zd\mu= \frac{M_{p,N}(z)}{M_{p,2}(z)}G(z)\mathbb E\|\xi\|^{p-z}$$ for
$max\{-1,p\}<\Re z<0.$  The right-hand side can be analytically continued to $max\{-1,p\}<\Re z<2.$  By equation (\ref{lim}) we have
$$ \lim_{r\to 2}\int_{\Omega}(g/f)^r f^p\,d\mu=0$$ which by Proposition \ref{extend} implies
$$ \int_{\Omega} g^2f^{p-2}\,d\mu=0$$ and this gives a contradiction.\end{proof}

\bigskip

\section{Examples}

We begin this section with some technical results which will be needed later.

\begin{lm} \label{bound} Let $X$ be a finite-dimensional normed space and suppose $\xi=\sum_{j=1}^m\gamma_jx_j$ is an $X$-valued Gaussian process, where $\{\gamma_1,\ldots,\gamma_n\}$ are independent normalized Gaussian random variables and each $x_j\neq 0.$  Then given $-n<u<0$ there is a constant $C=C(\xi,u)$ so that
$$ \mathbb E\|x+\xi\|^u \le C, \qquad x\in X.$$
\end{lm}

\begin{proof} We consider the case when $\xi$ is normalized so that $\mathbb E\|\xi\|^u=1.$  Let $E$ be the linear span of $\{x_1,\ldots,x_n\}$ and let $P$ be a projection of $X$ onto $E.$  Then
$$ \mathbb E\|x+\xi\|^u\le \|P\|^{-u}\mathbb E\|Px+\xi\|^u.$$
On $E$ the distribution $\mu_{\xi}$ is dominated by $C_0\lambda,$ where $C_0$ is a constant depending on $\xi$ and $\lambda$ is the Lebesgue measure on $E$.  Hence
$$ \mathbb E\|Px+\xi\|^u \le C_0\int_{\|e-Px\|\le 1}\|e-Px\|^ud\lambda(e) + 1$$ and this is uniformly bounded.\end{proof}

\begin{lm}\label{upper0}  Let $Z=X\oplus Y$ be a finite-dimensional normed space with $\dim X=m$ and $\dim Y=n.$  Suppose $\xi$ is a $Z$-valued Gaussian process of full rank.  Let $\xi_X,\xi_Y$ be the projections of $\xi$ onto $X$ and $Y$ respectively.  Then
\begin{equation}\label{upper}\mathbb E\|\xi_X\|^u\|\xi_Y\|^v<\infty, \qquad -m<u,\ -n<v.\end{equation}
\end{lm}

\begin{proof}  Note that $\xi_X$ and $\xi_Y$ are not necessarily independent.  However $\xi_X$ and $\xi_Y$  are of full rank in $X$ and $Y$ respectively.

If either $u=0$ or $v=0$ the Lemma holds trivially.  If either $u>0$ or $v>0$ we may use H\"older's inequality.  Suppose $v>0$.  Pick $a>1$ so that $au>-m$ and then suppose $1/a+1/b=1$.  Then
$$\mathbb E\|\xi_X\|^u\|\xi_Y\|^v\le (\mathbb E\|\xi_X\|^{au})^{1/a}(\mathbb E\|\xi_Y\|^{bv})^{1/b}<\infty.$$

Now suppose $u,v<0$.  We can write $\xi$ in the form
$$ \xi=\sum_{j=1}^{m+n}(x_j+y_j)\gamma_j$$ where $y_j=0$ for $n+1\le j\le m+n.$
Let $\mathbb E_0$ be the conditional expectation onto the $\sigma-$algebra $\Sigma$
generated by $\{\gamma_1,\ldots,\gamma_n\}.$  Then $\xi_Y$ is $\Sigma-$measurable.
Then, by Lemma \ref{bound}, since $\xi_X$ has rank $m$, there is a constant $C$ $$\mathbb E_0\|\xi_X\|^u\|\xi_Y\|^v=\|\xi_Y\|^v\mathbb E_0\|\xi_X\|^u\le C\|\xi_Y\|^v$$ and so \eqref{upper} holds.\end{proof}

\begin{lm}\label{existence} Suppose $1\le p<2.$ There exists a positive random variable $h$ with
$$ \mathbb E(h^z)= \frac{p}{2\Gamma(p/2)}\frac{\Gamma((p-z)/2)\Gamma(-z/2)}{\Gamma(-z/p)}, \qquad \Re z<2,$$ or
$$ \mathbb E(h^z) =\frac{2^{z/2}G(p-1-z)}{G(p-1)\Phi_{p/2}(z/2)}, \qquad \Re z<2.$$
\end{lm}

\begin{proof}
 Consider $f\approx\varphi_{1/p}^{\frac{1}{2p}}.$  Then by \eqref{phip}
 $$ \mathbb E(f^z) = \frac{p\Gamma(-z/2)}{\Gamma(-z/2p)}\qquad \Re z<2.$$
 If $f$ is defined on some probability space $(\Omega,\mathbb P)$ then we can consider $f$ as a random variable with respect to
 a new probability measure $$d\mathbb P'=\frac{\Gamma(1/2)}{p\Gamma(p/2)}|f|^{-p}d\mathbb P.$$ If we denote by $g$ this random variable we have
 $$ \mathbb E(g^z) =  \frac{\Gamma(1/2)}{\Gamma(p/2)}\frac{\Gamma((p-z)/2)}{\Gamma((p-z)/2p)}, \qquad  \Re z<p+2.$$

 Let $h\approx 2^{1/p}f\otimes g.$  Then for $\Re z<2$ and by using \eqref{99} we have
 \begin{align*}
 \mathbb E(h^z) &=\frac{p\Gamma(1/2)}{2\Gamma(p/2)}\frac{\Gamma(-z/2)\Gamma((p-z)/2)}{2^{-1-z/p}\Gamma(-z/2p)\Gamma((p-z)/2p)}\\
 &= \frac{p}{2\Gamma(p/2)}\frac{\Gamma(-z/2)\Gamma((p-z)/2)}{\Gamma(-z/p)}.\end{align*}
The second equation follows immediately from \eqref{GzG} and \eqref{phip}.\end{proof}

\begin{lm}\label{h} Suppose $m\in\mathbb N,$ and $\{p,q,r\}$ are such that $q>0$ and $p+m<q<r\le 2.$  There exists a positive random variable $g=g(m,p,q,r)$
such that
$$ \mathbb E(g^z)=\frac{2^{z/2}G(p+m-1-z)\Phi_{r/2}(z/2)\Phi_{r/2}((p-z)/2)}{G(p+m-1)\Phi_{r/2}(p/2)\Phi_{q/2}(z/2)}\qquad p-r<\Re z<p.$$
Here we adopt the convention that $\Phi_1(z)\equiv 1.$\end{lm}

\begin{proof}
We first use Lemma \ref{existence} to find a positive random variable $f_1$ such that
$$ \mathbb E(f_1^z)= \frac{2^{z/2}G(q-1-z)}{G(q-1)\Phi_{q/2}(z/2)}, \qquad \Re z<q.$$
Now, if $p+m<q$ we let $f_2$ to be distributed as $t^{-1/2}$ with respect to the Beta distribution
$$ d\mu=\frac{t^{(p+m)/2-1}(1-t)^{(q-p-m)/2-1}}{B((p+m)/2,(q-p-m)/2)}\,dt $$ on $[0,1].$
Then $$\mathbb E(f_2^z)= \frac{\Gamma((p+m-z)/2)\Gamma(q/2)}{\Gamma((q-z)/2)\Gamma((p+m)/2)}, \qquad \Re z<p+m,$$
and using \eqref{GzG} the latter can be rewritten as
$$ \mathbb E(f_2^z) = \frac{G(p+m-1-z)G(q-1)}{G(q-1-z)G(p+m-1)}, \qquad \Re z<p+m.$$  We write $f_2\equiv 1$ if $p+m=q.$
If $r<2$ we define $f_3\approx \varphi_{r/2}^{1/2}$ so that
$$ \mathbb E(f_3^z) =\Phi_{r/2}(z/2), \qquad \Re z<r.$$
If $r=2$ we set $f_3\equiv 1.$
If $f_3$ is defined on some probability space $(K,\mathbb P)$ we define $f_4$ as the random variable $f_3^{-1}$ with respect to the measure $f_3^pd\mathbb P/\mathbb E(f_3^p)$ so that $f_4\equiv 1$ if $r=2.$ If $r<2$ we have
$$ \mathbb E(f_4^z)= \frac{\Phi_{r/2}((p-z)/2)}{\Phi_{r/2}(p/2)}, \qquad p-r<\Re z.$$
We let $g\approx f_1\otimes f_2\otimes f_3\otimes f_4.$\end{proof}

\begin{lm}\label{g}\label{isotropicembedding}  Suppose $m\in\mathbb N,$ and $\{p,q,r\}$ are such that $q>0$ and $p+m<q<r\le 2.$ Suppose $Y\in\mathcal I_q.$  Then there is an $h$-isotropic embedding of $Y$ into $\mathcal M(\Omega,\mu)$ where $(\Omega,\mu)$ is a probability space where $h$ is symmetric and
$$ \mathbb E(|h|^z)=\frac{2^{z/2}G(p+m-1-z)G(z)\Phi_{r/2}(z/2)\Phi_{r/2}((p-z)/2)}{G(p+m-1)\Phi_{r/2}(p/2)}$$ for $ -1<\Re z<p+m.$ \end{lm}

\begin{proof}  Since $Y\in\mathcal I_q$ there is a $\psi_q-$isotropic embedding $S$ of $Y$ into some $\mathcal M(\Omega_1,\mu_1)$ (where $(\Omega_1,\mu_1)$ is a probability measure space).  Let $Ty=2^{-1/2}gSy$ where $g$ is independent of $S(Y)$ and distributed as in Lemma \ref{h}.  Then $T$ is a $h$-isotropic embedding where $h$ is symmetric and
$$ \mathbb E(|h|^z)=2^{-z/2}\Psi_q(z)\mathbb E(g^z)=\Phi_{q/2}(z/2)G(z)\mathbb E(g^z)\qquad -1<\Re z<p+m.$$ \end{proof}

Let us remark that the case $p=0,$ $m=1$ and $r=2$ gives
$E(|h|^z)= 2^{z/2}G(z)G(-z)$ which means that $h$ is symmetric 1-stable, i.e. has the Cauchy distribution.

\begin{Thm}\label{mainexample}  Suppose $1\le q,r\le 2$.  Suppose $X=\ell_2^m$ and $Y\in\mathcal I_q.$  If $p\le q-m$ then $X\oplus_rY\in \mathcal
I_p.$\end{Thm}

\begin{proof} It is enough to consider the case $p+m>0.$  Also, the result holds trivially if $r\le q$ since then $X\oplus_rY\in \mathcal I_r\subset \mathcal I_p.$  So we may also assume that $r>q.$ Hence $p-r<q-1-r<-1.$

We treat three separate cases as $p>0,$ $p<0$ or $p=0.$

{\it Case 1: } Let $p>0$.  In this case we have $m=1$ and identify $X$ with $\mathbb R.$ In view to Lemma \ref{g} we construct an $h$-isotropic embedding $S:Y\to\mathcal M(\Omega,\mu)$ where $h$ is symmetric and
\begin{equation}\label{H}  H(z):=\mathbb E(|h|^z)= \frac{G(p-z)G(z)\Phi_{r/2}(z/2)\Phi_{r/2}((p-z)/2)}{G(p)\Phi_{r/2}(p/2)}\end{equation} for $ -1<\Re z<p+1.$  It is important to observe that $H(p)=1$ and
\begin{equation}\label{H1}  H(z)= \frac{G(p-z)G(z)M_{p,r}(z)}{G(p)M_{p,2}(z)}\end{equation}
for $ -1<\Re z<p+1.$

We define
$T:\mathbb R\oplus_rY\to\mathcal M(\Omega,\mu)$ by $T(\alpha,y)=\alpha+Sy.$
To verify that $T$ is a standard isometry we only need to show (considering $h$ as a function on $(\Omega,\mu)):$
\begin{equation}\label{aim}\int_{\Omega}|1+th|^p\,d\mu =(1+t^r)^{p/r}, \qquad 0<t<\infty.\end{equation}

To establish \ref{aim} we call Proposition \ref{Mellinh}.  By \eqref{eqMellinh0} and \eqref{H1} we have
$$ \int_0^{\infty}t^{-z-1}\int_{\Omega}\left(|1+th|^p-\min\{1,t\}^p\right)d\mu\,dt = M_{p,r}(z)+\frac{p}{(p-z)z}$$ for $-1<\Re z<p+1.$

On the other hand, by \eqref{102},\eqref{Fq} and \eqref{Mpq}
$$ \int_0^{\infty}t^{-z-1}\bigl((1+t^r)^{(w+z)/r}-\max\{1,t\}^{w+z}\bigr)\,dt= F_r(w,z)-F_\infty(w,z), $$ for $\Re w,\ \Re z<0$ and by analytic continuation this holds (and the right-hand side is holomorphic) for
$\Re w,\ \Re z<r.$ Thus using \eqref{Minf}
$$ \int_0^{\infty}t^{-z-1}((1+t^r)^{\frac{p}{r}}-\max\{1,t\}^p)\,dt = M_{p,r}(z)+\frac{p}{z(p-z)}, \qquad 0<\Re z <p$$ and by the uniqueness property of the Mellin transform, Proposition \ref{Mellin}, we conclude
that
$$ \int |1+th|^p\,d\mu=(1+t^r)^{p/r}, \qquad 0<t<\infty$$ which proves the Theorem for $p>0.$

{\it Case 2:} Let $p<0 $ and let $\dim Y=n.$ This is quite similar but now we deal with Gaussian embeddings rather than standard embeddings.
First we note that there is an $f$-isotropic embedding of $\ell_2^m$ into $\mathcal M(\Omega,\mu),$ where $(\Omega,\mu)$ is a probability measure space and $f$ is symmetric with
$$ \int |f|^z\,d\mu = \frac{G(z)G(m-1)}{G(z+m-1)}, \qquad -m<\Re z<\infty.$$  Indeed let $\{\gamma_1,\ldots,\gamma_m\}$ be independent normalized Gaussian random variables and let
$$ R(a_1,\ldots,a_m)=\frac{a_1\gamma_1+\cdots+a_m\gamma_m}{(\gamma_1^2+\cdots+\gamma_m^2)^{1/2}}.$$  We now use Lemma \ref{1}.
We consider $h=h(m,p,q,r)$ as in Lemma \ref{g}. Then
$$ H(z):= \mathbb E(|h|^z)=\frac{G(p+m-1-z)G(z)\Phi_{r/2}(z/2)\Phi_{r/2}((p-z)/2)}{G(p+m-1)\Phi_{r/2}(p/2)}$$ for $ -1<\Re z<p+m.$   Note that by \eqref{phip} and \eqref{Fq}
\begin{equation}\label{HG} H(z)= \frac{G(p+m-1-z)G(z)F_r(p-z,z)}{G(p+m-1)F_2(p-z,z)}.
\end{equation}

  We may then suppose that $S:Y\to\mathcal M(\Omega,\mu)$ is an $h$-isotropic embedding such that $R(X)$ and $S(Y)$ are independent.
Finally we define $T:X\oplus_rY\to \mathcal M(\Omega,\mu)$ by
$$ T(x+y)= \theta(Rx+Sy)$$ where $\theta>0$ is chosen so that $$\theta^p=\frac{G(p+m-1)}{G(m-1)}.$$  We will show that $T$ is a 1-Gaussian embedding.  To do this we suppose that $\xi$ is an $X\oplus_r Y-$valued Gaussian process of full rank.
Let $P:X\oplus_r Y\to X$ and $Q:X\oplus_rY\to Y$ be the natural projections onto $X$ and $Y$ respectively. Let $\xi_X=P\xi$ and $\xi_Y=Q\xi.$  Then $\xi_X$ has full rank on $X$ and $\xi_Y$ on $Y.$  In particular we can write $\xi=\sum_{j=1}^{m+n}(x_j+y_j)\gamma_j$ where $y_j=0$ for $n+1\le j\le m+n,$
by choosing an appropriate basis of Gaussian random variables.

For $0<s<t$ we have
$$\mathbb E\|\xi_X+t\xi_Y\|^p\le \mathbb E\|\xi_X+s\xi_Y\|^p \le (s/t)^p\mathbb E\|\xi_X+t\xi_Y\|^p, \qquad .$$
So, the function $t\mapsto \mathbb E\|\xi_X+t\xi_Y\|^p$ is continuous on $(0,\infty).$

Similarly since $\xi_X$ has full rank, $\{x_{n+1},\ldots,x_{m+n}\}$ form a basis of $X.$ This implies
$$ \sum_{j=n+1}^{m+n} |Rx_j|^2 \ge c^2>0 \qquad \text{a.e.}$$ and thus, since $p<0$
$$ (\sum_{j=1}^{m+n}|T(x_j+ty_j)|^2)^{p/2} \le c^p\qquad \text{a.e.}$$  We now may conclude, by the Lebesgue Dominated Convergence Theorem, that the map
 $$ t\mapsto \int_{\O}(\sum_{j=1}^{m+n}|T(x_j+ty_j)|^2)^{p/2}d\mu$$ is also continuous on $(0,\infty)$.  We will show that
 \begin{equation}\label{needed}\int_{\O}(\sum_{j=1}^{m+n}|T(x_j+ty_j)|^2)^{p/2}d\mu =\mathbb E\|\xi_X+t\xi_Y\|^p, \qquad 0<t<\infty\end{equation}
 by computing the Mellin transform of the left and the right-side of the equality.

 By Lemma \ref{upper0} we have \begin{equation}\label{upper1}\mathbb E\|\xi_X\|^u\|\xi_Y\|^v<\infty, \qquad -m<u,\ -n<v.\end{equation}

 Suppose $x\in X$ and $y\in Y$ are non-zero.  Then for $-1/2<u,\ v<0$ we use Lemma  \ref{independent} to compute:
\begin{align*} &\int_0^{\infty}\int_{\Omega}t^{-1-v}|T(x+ty)|^{u+v}\,d\mu\,dt=\\
&=\frac12\int_{\Omega}\int_0^{\infty}t^{-1-v}(|Tx+tTy|^{u+v}+|Tx-tTy|^{u+v})\,dt\,d\mu\\
&=  \frac{G(u+v)F_2(u,v)}{G(u)G(v)}\int_{\Omega} |Tx|^u|Ty|^v\,d\mu .\end{align*}
Then by the definition of $T,$ equation \eqref{HG} and Lemma \ref{1}, the latter is equal to
$$ G(m-1)\frac{\theta^{u+v}G(u+v)F_2(u,v)H(v)}{G(u+m-1)G(v)}\|x\|^u\|y\|^v<\infty. $$
The calculation can then be repeated for $-1/2<\Re w, \Re<0$ to give
 \begin{align*} &\int_0^{\infty}\int_{\Omega}t^{-z-1}|T(x+ty)|^{w+z}\,d\mu\,dt \\ &=G(m-1)\frac{\theta^{w+z}G(w+z)F_2(w,z)H(z)}{G(w+m-1)G(z)}\|x\|^w\|y\|^z.\end{align*}

 Again calculating first with real $u,v,$ using Tonelli's theorem and since $\{\gamma_i\}$ are Gaussian r.v. we may compute the following integral for $-1/2<\Re z,\ \Re w<0,$
 \begin{align*} &\int_0^\infty t^{-z-1} \int_\Omega (\sum_{j=1}^{m+n} (T(x_j+ty_j))^2)^{(w+z)/2}d\mu\,dt\\
&= \frac{1}{G(w+z)}\mathbb E\left(\int_0^{\infty}t^{-z-1}(\int_\Omega |\sum_{j=1}^{m+n}\gamma_jTx_j +t\gamma_jTy_j|^{w+z}d\mu)dt\right).\end{align*}
Then by Lemma \ref{symmetrize}, using \eqref{upper1} we have
\begin{align}\label{MT} &= G(m-1)\frac{\theta^{w+z}F_2(w,z)H(z)}{G(w+m-1)G(z)}\mathbb E\left(\|\sum_{j=1}^{m+n}\gamma'_jx_j\|^w\|\sum_{j=1}^{m+n}\gamma'_jy_j\|^z\right) \nonumber \\
&= G(m-1)\frac{\theta^{w+z}F_2(w,z)H(z)}{G(w+m-1)G(z)}\mathbb E(\|\xi_X\|^w\|\xi_Y\|^z)
\end{align}
Now the right-hand side of \eqref{MT} extends to be holomorphic when $-n<\Re z<0$ and $-m<\Re w<0.$  Using Proposition \ref{extension} (twice) one obtains that
$$\int_0^{\infty}\int_{\Omega}t^{-1-v}(\sum_{j=1}^{m+n} (T(x_j+ty_j))^2)^{(u+v)/2}\,d\mu\,dt<\infty$$ when $-m<u<0$ and $-n<v<0.$   This in turn means that the function
$$(w,z)\mapsto \int_0^{\infty}\int_{\Omega}t^{-z-1}(\sum_{j=1}^{m+n} (T(x_j+ty_j))^2)^{(w+z)/2}\,d\mu\,dt $$ is holomorphic for $-m<\Re w<0$ and $-n<\Re z<0.$  Thus we have
 \begin{align*} &\int_0^{\infty}\int_{\Omega}t^{-z-1}(\sum_{j=1}^{m+n} (T(x_j+ty_j))^2)^{(w+z)/2}\,d\mu\,dt=\\ &=G(m-1)\frac{\theta^{w+z}F_2(w,z)H(z)}{G(z)G(w+m-1)}\mathbb E\|\xi_X\|^w\|\xi_Y\|^z\end{align*} whenever $-m<\Re w<0$ and $-n<\Re z<0.$
 In particular by \eqref{HG} we have that for $\max\{-n,p\}<\Re z<0$
 \begin{equation}\label{mel}\int_0^{\infty}t^{-z-1}\int_{\Omega}(\sum_{j=1}^{m+n} (T(x_j+ty_j))^2)^{p/2}\,d\mu\,dt=M_{p,r}(z)\mathbb E\|\xi_X\|^{p-z}\|\xi_Y\|^z.\end{equation}

 On the other hand for $-1/2<\Re w,\ \Re z<0$, using \eqref{Fq} we get that
 \begin{align*} \int_0^{\infty}t^{-1-z}E\|\xi_X+t\xi_Y\|^{w+z}dt &=\mathbb E\int_0^{\infty}t^{-z-1}(\|\xi_X\|^r+t\|\xi_Y\|^r)^{(w+z)/r}dt\\
 &=F_r(w,z) \mathbb E\|\xi_X\|^w\|\xi_Y\|^z.\end{align*}
 As before these calculations should be done first for real $w,z$ to justify the use of Fubini's theorem.
 Since the right-hand side is holomorphic for $-m<\Re w<0$ and $-n<\Re z<0$, we again use Proposition \ref{extension} to derive equality for $(w,z)$ in the larger region.  Hence the Mellin transform of the right-side of \eqref{needed} is
\begin{equation}\label{ME} \int_0^{\infty}t^{-1-z}E\|\xi_X+t\xi_Y\|^pdt=M_{p,r}(z)\mathbb E\|\xi_X\|^{p-z}\|\xi_Y\|^z,
\end{equation} for $\max\{-n,p\}<\Re z<0.$  Comparing \eqref{mel} and \eqref{ME} we get \eqref{needed}. In particular
 $$ \mathbb E\|\xi\|^p=\int_{\O} (\sum_{j=1}^{m+n}|T(x_j+y_j)|^2)^{p/2}\,d\mu,$$  which implies that $T$ is a 1-gaussian embedding.

 {\it Case 3:} When $p=0$ the result follows by showing that each space embeds into $L_p$ for every $p<0.$
 \end{proof}

 The particular case $p=0$ with $r=2$ also follows if we consider Proposition 6.6 of \cite{KaltonKoldobskyYaskinYaskina2007}.

 \begin{Thm}\label{ex2}  For any $-\infty<p<2$ and any $n\ge 3-p$ there exists a normed space $X$ of dimension $n$ such that $X\in \mathcal I_s$ whenever $s\le p$ and $X\notin \mathcal I_s$ whenever $s>p.$  We may take  $X=\ell_2^{1-[p]}\oplus_r\ell_{q}^{n-1+[p]}$ where $q=1+p-[p]$ and $q<r\le 2.$\end{Thm}

 \begin{proof}  If $1\le p<2$ then $q=p$ and $X=\ell_p^n.$ Then by \cite{Koldobsky1999b} we have that $X\in\mathcal I_s$ only if $s\le p$ (see also the Introduction).

 Let $p<1.$ Then by Theorem \ref{mainexample} if $s\le p$ then $X\in\mathcal I_s,$ since $q=m+p.$ Conversely, we suppose that $X\in\mathcal I_s.$ If $n=1$ there is nothing to prove, so we may assume that $n\geq2.$ Then $n-1+[p]\ge 2-p+[p]>1.$ By Theorem \ref{main}, $\ell_q^{n-1+[p]}\in \mathcal I_{\alpha},$ where $\a\leq min\{r, s+1-[p]\}.$ But $q<r$ so $\ell_q^{n-1+[p]}\in \mathcal I_{s+1-[p]}.$ Consequently, $s+1-[p]\le 1+p-[p]$ which implies that $s\le p.$ \end{proof}

\end{document}